\documentclass[a4paper, 11pt]{amsart}
\usepackage[T1]{fontenc}
\usepackage{amssymb,amsthm,amsmath}
\usepackage{enumerate}
\usepackage{eucal}

\newtheorem{thm}{Theorem}[section]
\newtheorem{prop}[thm]{Proposition}
\newtheorem{lem}[thm]{Lemma}
\newtheorem{cor}[thm]{Corollary}

\newtheorem{fact}[thm]{Fact}

\theoremstyle{definition}
\newtheorem{defi}[thm]{Definition}
\newtheorem{exam}[thm]{Example}
\newtheorem{rem}[thm]{Remark}

\numberwithin{equation}{section}
\newcommand{\mb}{\mathbb}

\DeclareMathOperator{\topo}{top}
\DeclareMathOperator{\bor}{bor}

\newcommand{\lss}{\mathbf{LSS}}

\newcommand{\ads}{\mathbf{ADS}}
\newcommand{\ds}{\mathbf{DS}}
\newcommand{\lds}{\mathbf{LDS}}

\def\mc{\mathcal}
\def\mb{\mathbb}

\begin{document}
\title[Locally small spaces]
{Locally small spaces with an application}
\author[A. Pi\k{e}kosz]{Artur Pi\k{e}kosz}
\address{Institute of Mathematics, Cracow University of Technology, Warszawska 24, 31-155 Krak\'ow,  Poland}
\email{pupiekos@cyfronet.pl}

\thanks{}


\date{\today}

\begin{abstract}
We develop the theory of  locally small spaces in a new simple language
and apply this simplification  to re-build the theory of locally definable spaces over
 structures with topologies.

\end{abstract}

\maketitle

\section{Introduction.}
Topological structures play a very important role in mathematics. Today the classical topology (whose fathers were, among others,  H. Poincar\'{e}, F. Hausdorff and K. Kuratowski)  is known by every mathematician.  
Generalizations of the classical notion of a topological space have been invented in many directions by E. \v{C}ech, H. Herrlich, M. Kat\v{e}tov, F. Riesz and many others.

Grothendieck topologies, introduced in the beginning of the 1960s by A. Grothendieck (see \cite{A} and \cite{AGV}), use the language of category theory, but still are loosely considered as generalizations of classical topologies. They became very   fruitful in algebraic geometry and rigid analytic geometry.
S. Bosch, U. G\"{u}ntzer and R. Remmert (\cite{BGR}) used concrete Grothen\-dieck topologies (called by them $G$-topologies).  A $G$-topology on a set $X$ consists of a family of admissible open subsets of $X$ and a family of admissible coverings for each such admissible open set that satisfy  several conditions (\cite{BGR}, Definition 1 in Subsection 9.1.1).
Sometimes, three additional conditions $(G_0)$, $(G_1)$, $(G_2)$ are assumed (\cite{BGR}, Subsection 9.1.2, page 339).

H. Delfs and M. Knebusch (\cite{DK}) added to the list of required properties two important ones: 
open (admissible) sets should be closed under finite unions and all finite coverings of open (admissible) sets should be admissible. This way, they introduced their generalized topological spaces. 
Then they developed a semialgebraic version of homotopy theory, which was extended by A. Pi\k{e}kosz (\cite{Pie1}) to a homotopy theory over  o-minimal  expansions of  fields. 

\'{A}. Cs\'{a}sz\'{a}r,  besides his syntopogeneous structures,  introduced another 
notion of a  generalized topology (\cite{C}), where the family of open sets 
satisfies some, but not all, conditions required for a topology.

We define  a locally small space in a similar way: the distinguished family
of subsets of the underlying set is required to satisfy some conditions a topology satisfies.
Still our locally small spaces can be seen as a special kind of Delfs-Knebusch generalized topological spaces, so, in particular, a special kind of
sets with $G$-topologies on them.
In this paper, a {locally small space} is a pair $(X,\mc{L}_X)$, where $X$ is any set and 
$\mc{L}_X$ is a subfamily of the powerset $\mc{P}(X)$ closed under finite unions and finite intersections, containing the empty set $\emptyset$ and  covering $X$.

In Section 2, we develop the theory of locally small spaces in this simplified language.
The main results are Theorems \ref{1ki}, \ref{subsp},  \ref{2ki}, \ref{refl} and \ref{korefl}. We also give two generalizations of facts known from \cite{DK} in Theorems 
\ref{cpar} and \ref{stlind}.  
As an application, in Section 3, the categories of locally definable spaces over structures with topologies are re-built using the new language of locally small spaces.
This means that all results of the monograph \cite{DK} and of the paper \cite{Pie1} about locally definable spaces stay valid when our approach replaces the approach of H. Delfs and M. Knebusch with a longish and complicated definition of a generalized topological space.   

\textbf{Notation.} For a set $X$, its powerset is denoted by $\mc{P}(X)$.
 
We shall use a special notation for operations on families of sets, for example for family intersection
$$ \mc{U} \cap_1 \mc{V} = \{ U \cap  V : U \in \mc{U}, V \in \mc{V} \},$$ 
and sometimes for families of families, namely
$$ \Phi \cap_2 \Psi = \{ \mc{U} \cap_1 \mc{V} : \mc{U} \in \Phi, \mc{V} \in \Psi \}.$$

\section{Locally small spaces}
\subsection{Basic definitions and examples}
\begin{defi}\label{lss}
A \textbf{locally small space} $\mc{X}$  is a pair $(X,\mc{L}_X)$, where $X$ is any set and  $\mc{L}_X\subseteq  \mc{P}(X)$ satisfies the following conditions:\\
(LS1) $\emptyset \in \mc{L}_X$,\\
(LS2) if $L,M \in \mc{L}_X$, then $L\cap M, L\cup M \in \mc{L}_X$,\\
(LS3) $\forall x \in X \: \exists L_x \in \mc{L}_X  \:  x\in L_x$ (i. e. $\bigcup \mc{L}_X = X$).

Elements of $\mc{L}_X$ are called \textbf{small open} subsets of $X$, \textbf{smops} for short. (Locally small spaces will sometimes be shortly  called spaces.) 
\end{defi}

\begin{defi}
Let a family $\mc{A}$ of subsets of $X$ be given.  
Define a new family
$$\mc{A}^o=\{ Y\subseteq X \: |\:  Y \cap_1 \mc{A} \subseteq \mc{A} \}.$$
Elements of $\mc{A}^o$ will be called the sets \textbf{compatible with} the family 
$\mc{A}$. 
\end{defi}

\begin{lem}
Assume a family $\mc{A}\subseteq \mc{P}(X)$ is closed under finite intersections (i. e.  $\mc{A} \cap_1 \mc{A} \subseteq \mc{A}$). Then  
$$  \mc{A}\subseteq \mc{A}^o \mbox{ and }    \mc{A}^{oo}=\mc{A}^o. $$
\end{lem}
\begin{proof} The first part in obvious.

Assume $W\in \mc{A}^{oo}$. Take $A\in \mc{A}$. Then $W\cap A\in \mc{A}^{o}$.
But $W\cap A=(W\cap A) \cap A \in \mc{A}$. This means $W\in \mc{A}^{o}$.

Now assume $W\in \mc{A}^{o}$. Take any $B\in \mc{A}^{o}$. Then $W\cap B \in \mc{A}^{o}$.  This means $W\in \mc{A}^{oo}$.
\end{proof}

From now on, we assume that a locally small space  $(X, \mc{L}_X)$ is given.
\begin{defi}
The family $\mc{L}^{o}_X$ of all \textbf{open} subsets of $X$ is the family of all subsets of $X$ compatible with smops:
$$\mc{L}^{o}_X=\{ V \subseteq X \: | \: V \cap_1 \mc{L}_X \subseteq \mc{L}_X \}.$$
\end{defi}

\begin{rem}
By (LS2),  we always have  $\mc{L}_X\subseteq \mc{L}^o_X$.
\end{rem}

\begin{defi}[cf. \cite{AHS}, Ex. 22.2(2) and \cite{Pie3}, Def.  2.2.29]
A \textbf{bornology} on a set $X$ is a family $\mc{B}\subseteq \mc{P}(X)$ such that:
\begin{enumerate}
\item if $A,B \in \mc{B}$, then $A \cup B \in \mc{B}$,
\item if $B \in \mc{B}$ and $A\subseteq B$, then $A \in \mc{B}$.
\end{enumerate}
\end{defi}

\begin{defi}
The family  $\mc{L}^s_X$  of all \textbf{small} subsets of $X$  is the smallest bornology on $X$ containing $\mc{L}_X$:
$$ \mc{L}^s_X=\bor(\mc{L}_X)=  \{ B\subseteq X \: |\:  B\subseteq L \mbox{ for some }L \in \mc{L}_X \}.$$ 
\end{defi}

\begin{prop}
We have $\mc{L}_X = \mc{L}^s_X \cap \mc{L}^o_X$.
\end{prop}
\begin{proof}
Clearly, $\mc{L}_X \subseteq \mc{L}^s_X$ and  $\mc{L}_X \subseteq \mc{L}^o_X$.
Assume $V \in  \mc{L}^s_X \cap \mc{L}^o_X$. Then there exists some $W \in \mc{L}_X$ such that $V \subseteq W$, hence $V=V\cap W \in \mc{L}^o_X \cap_1 \mc{L}_X \subseteq \mc{L}_X$.
\end{proof}

\begin{defi}
The family  $\mc{L}^{wo}_X$ of all \textbf{weakly open}  subsets of $X$ is the smallest topology on $X$ containing 
$ \mc{L}_X $, so is given by the formula
$$  \mc{L}^{wo}_X = \topo(\mc{L}_X)=\{ \bigcup \mc{U} : \mc{U} \subseteq  \mc{L}_X  \}.$$ 
\end{defi}

\begin{prop}
We always have $\mc{L}^o_X \subseteq \mc{L}^{wo}_X$. 
\end{prop}
\begin{proof}
If $V  \in  \mc{L}^o_X$, then $V$ is the union of the family $V\cap_1 \mc{L}_X$, which is a subfamily of  $\mc{L}_X$, so  $V\in \mc{L}^{wo}_X$.
\end{proof}

\begin{defi}
The family of all \textbf{small weakly open} subsets of $X$ is the family
$$\mc{L}^{swo}_X =  \mc{L}^{wo}_X \cap \mc{L}^s_X .$$
\end{defi}

\begin{prop} We always have $\mc{L}_X \subseteq  \mc{L}^{swo}_X$.
\end{prop}
\begin{proof}
We have $\mc{L}_X = \mc{L}^s_X \cap \mc{L}^o_X \subseteq  \mc{L}^s_X \cap \mc{L}^{wo}_X = \mc{L}^{swo}_X$.
\end{proof}

\begin{defi}[cf. \cite{Pie2}, Def. 2.2.8]
The \textbf{weak closure} $wcl(Y)$ of a set $Y\subseteq X$ is the closure of $Y$ in the topological space $(X,\mc{L}_X^{wo})$.
\end{defi}

Many of generalized topological spaces in the sense of Delfs and Knebusch on the real line mentioned in Definition 1.2 of \cite{PW} are locally small. 
We can restate their definitions to get locally small spaces in the sense of Definition \ref{lss} above.
\begin{exam}\label{real-lines}
 For $X=\mathbb{R}$, we can take as the family of smops any of the following families (where the words ``locally'' and ``bounded'' are understood traditionally):\\
(1) $\mc{L}_{om}=$ the finite unions of open intervals  
(we get $\mathbb{R}_{om}$), \\
(2) $\mc{L}_{rom}=$ the finite unions of open intervals with rational endpoints  (we get $\mathbb{R}_{rom}$), \\
(3)   $\mc{L}_{lom}=$ the finite unions of bounded open intervals  (we get $\mathbb{R}_{lom}$),\\
(4)  $\mc{L}_{l^+om}=$ the finite unions of bounded from above open intervals (we get $\mathbb{R}_{l^+om}$),\\
(5)   $\mc{L}_{slom}=$ the locally finite unions of bounded open intervals  (we get $\mathbb{R}_{slom}$),\\
(6)   $\mc{L}_{sl^+om}=$ the locally  finite unions of bounded open intervals that are finite unions on the negative halfline  (we get $\mathbb{R}_{sl^+om}$),\\
(7) $\mc{L}_{st}= $ the natural topology (we get  $\mathbb{R}_{st}$) ,\\
(8)   $\mc{L}_{lst}= $ the bounded sets from the natural topology  (we get  $\mathbb{R}_{lst}$),\\
(9)  $\mc{L}_{l^+st}=$ the bounded from above sets from the natural topology  (we get $\mathbb{R}_{l^+st}$).

On the other hand, the space $\mb{R}_{ut}$ from Definition 1.2 of \cite{PW} is not locally small.
\end{exam}

We shall give a definition of a locally small (Delfs-Knebusch) generalized topological space in Subsection 2.5.

\subsection{Admissible families}
\begin{defi}
Any subfamily  of $\mc{L}_X^o$  will be called an \textbf{open family}.

We say that an open family  $\mc{U}$ is:
\begin{enumerate}
\item[(a)] \textbf{essentially finite} if some finite subfamily $\mc{U}_f \subseteq \mc{U}$ covers the union of~$\mc{U}$ (i. e. $\bigcup \mc{U} = \bigcup \mc{U}_f$), 
\item[(b)] \textbf{locally finite} if each member of $ \mc{L}_X$ has a non-empty intersection with  only finitely many members of $\mc{U}$, 
\item[(c)]   \textbf{admissible}  (or  \textbf{locally essentially finite}) if for each $L\in \mc{L}_X$ there exists a finite subfamily $\mc{U}_L\subseteq \mc{U}$ such that $(\bigcup \mc{U}) \cap L = (\bigcup \mc{U}_L) \cap L$.
\end{enumerate}
\end{defi}

\begin{rem}
Notice that the word ``locally'' has a special meaning in the theory of locally small spaces. 
\end{rem}

\begin{prop}[cf. \cite{Pie3}, Cor. 2.1.19]
Each locally finite open family  in a locally small space  is admissible.
\end{prop}
\begin{proof}
For $\mc{U}$ locally finite and $L\in \mc{L}_X$, take $\mc{U}_L=\{ V \in \mc{U} : V \cap L \neq \emptyset \}$.
\end{proof}

\begin{prop}
The family  $\mc{L}_X$ is admissible in $(X, \mc{L}_X)$.
\end{prop}
\begin{proof}
For $\mc{U}=\mc{L}_X$ and $L\in \mc{L}_X$, take $\mc{U}_L=\{ L\}$.
\end{proof}

\begin{exam}
In the space $\mb{R}_{om}=(\mb{R},\mc{L}_{om})$ from Example \ref{real-lines}, 
the open and pairwise disjoint  family $\mc{U}=\{(k,k+1): k \in \mb{Z} \}$ is neither locally finite nor admissible  and its union is not open (but only weakly open).
\end{exam}

\subsection{Smallness and compactness}

\begin{thm} 
For a locally small space $(X,\mc{L}_X)$, the following conditions are equivalent:
\begin{enumerate}
\item[$(1)$]  $X\in \mc{L}_X$, 
\item[$(2)$]  $\mc{L}^s_X=\mc{P}(X)$, 
\item[$(3)$]  $\mc{L}_X = \mc{L}^o_X$, 
\item[$(4)$] \: each admissible family is essentially finite, 
\item[$(5)$] \: each admissible covering of $X$ is essentially finite.
\end{enumerate}
\end{thm}
\begin{proof}
$(1) \Rightarrow (2)$ Since $X\in \mc{L}_X \subseteq \mc{L}^s_X$, then $\mc{L}^s_X=\mc{P}(X)$.\\
$(2) \Rightarrow (3)$ We have 
$\mc{L}_X = \mc{L}^s_X \cap \mc{L}^o_X=\mc{L}^o_X$. \\
$(3) \Rightarrow (1)$ We have $X\in \mc{L}^o_X=\mc{L}_X$.\\
$(1) \Rightarrow (4)$ If $\mc{U}$ is an admissible family, then for $L=X$ there exists a finite subfamily  $\mc{U}_X\subseteq \mc{U}$ such that
$\bigcup \mc{U}=\bigcup \mc{U}_X$. So $\mc{U}$ is essentially finite.\\
$(4) \Rightarrow (5)$ Trivial.\\
$(5) \Rightarrow (1)$ The admissible covering $\mc{L}_X$ of $X$   is essentially finite. There exist $L_1,...,L_k \in \mc{L}_X$ such that $X=\bigcup \mc{L}_X = L_1 \cup
...\cup L_k$. Hence $X\in \mc{L}_X$.
\end{proof}

\begin{defi}
A locally small space  $(X, \mc{L}_X)$
is \textbf{small} if it satisfies one of the conditions of the previous theorem.
\end{defi}
 
\begin{defi}[cf. \cite{PW}, Def. 3.2]
A locally small space $(X, \mc{L}_X)$  is \textbf{topologically compact} if the topological space $(X, \mc{L}^{wo}_X)$ is compact (i. e. each subfamily of $\mc{L}^{wo}_X$ covering $X$ admits a finite subcovering).
\end{defi}

\begin{rem}
Topological compactness implies smallness, but smallness does not imply topological compactness (take $\mb{R}_{st}, \mb{R}_{slom}, \mb{R}_{om}$ or $\mb{R}_{rom}$ from Example \ref{real-lines}).
In the class of locally small spaces, admissible compactness (i. e. condition (5) above, 
see also  Definitions 3.2 and 3.3 of \cite{PW}) is equivalent to smallness.
\end{rem}

\subsection{Gluing of spaces}

\begin{defi}
An \textbf{open subspace}  of $(X, \mc{L}_X)$ induced by  $U\in \mc{L}^o_X$ is a pair   of the form $(U, \mc{L}_X \cap_1 U)$. 
\end{defi}

\begin{defi}[cf. \cite{Pie2}, Def. 2.2.43]
For a family of locally small spaces $\{(X_i,\mc{L}_i)\}_{i\in I}$ such that
\begin{enumerate}
\item[$(\star )$] 
each $X_i\cap X_j  (i,j\in I)$ belongs both to $\mc{L}^o_i$ and to $\mc{L}^o_j$ and
the open subspaces induced by  $X_i\cap X_j$ both in $(X_i, \mc{L}_i)$ and in  
$(X_j, \mc{L}_j)$ are equal 
(i. e. $\mc{L}_i  \cap_1 X_j = \mc{L}_j  \cap_1 X_i$),
\end{enumerate}
the  \textbf{admissible union} of this family is the pair $(X, \mc{L}_X)$, 
where $X=\bigcup_{i\in I} X_i$ and $\mc{L}_X\subseteq \mc{P}(X)$ is the smallest ring of sets containing  $\bigcup_{i\in I}\mc{L}_i$. 
We shall then  write 
$$(X,\mc{L}_X)=\bigcup^a_{i \in I} (X_i,\mc{L}_i).$$
\end{defi}

\begin{prop}\label{glu}
For an admissible union  as above:
\begin{enumerate}
\item[$(a)$] 
members of $\mc{L}_X$ are sets of the form $L=L_1 \cup ... \cup L_k$, where
$L_j \in \mc{L}_{i_j}$ for $j=1,...,k$, 
\item[$(b)$] each $(X_i,\mc{L}_i)$ is an open subspace of $(X, \mc{L}_X)$, 
\item[$(c)$] the family $\{X_i\}_{i\in I}$ is admissible in $(X, \mc{L}_X)$.
\end{enumerate}
\end{prop}
\begin{proof}
$(a)$: It is sufficient to prove that for $L_i\in \mc{L}_i$ and $L_j\in  \mc{L}_j$ we have $L_i\cap L_j \in  \mc{L}_i \cap  \mc{L}_j$. We have 
$$L_i\cap L_j=(L_i\cap X_j) \cap (X_i \cap L_j)\stackrel{(\star )}{=} (L_i\cap X_j) \cap (\tilde{L}\cap X_j)= (L_i\cap \tilde{L}) \cap X_j \in  \mc{L}_i,$$
where $\tilde{L}\in \mc{L}_i$.  We get $L_i\cap L_j \in  \mc{L}_j$ similarly. \\
$(b)$: We check that $X_i \in  \mc{L}^o_X$.
Take $L \in \mc{L}_X$. Then $L=L_1 \cup ... \cup L_k$, 
where
$L_j \in \mc{L}_{i_j}$ for $j=1,...,k$. 
Since $X_i \cap L_j \in \mc{L}_{i_j}$, we get $X_i \cap L \in  \mc{L}_X$.
Hence $X_i \in  \mc{L}^o_X$.

We check that $ \mc{L}_i =  \mc{L}_X \cap_1 X_i$. 
By the definition of  $\mc{L}_X$, we get $ \mc{L}_i \subseteq  \mc{L}_X \cap_1 X_i$. For the opposite inclusion,   assume $L\in \mc{L}_X$ and $L=\bigcup_{j=1}^k L_j$, where  $L_j \in \mc{L}_{i_j}$ for $j=1,...,k$. 
Then $L\cap X_i=\bigcup_{j=1}^k L_j \cap X_i\stackrel{(\star )}{=} 
\bigcup_{j=1}^k X_{i_j} \cap M_{j_i} \cap X_i$ with some $M_{j_i} \in \mc{L}_i$.
Since $X_{i_j}\cap X_i \in  \mc{L}^o_i$,  we get $L\cap X_i \in \mc{L}_i$.\\
$(c)$: Take $L\in \mc{L}_X$. Then $L=\bigcup_{j\in J} L_j$ $(L_j\in \mc{L}_{i_j})$, where $J$ is finite, so  the family $\{X_i\}_{i\in I} \cap_1 L=\{ X_i \cap L\}_{i\in I}$ 
 is essentially finite, since $L=\bigcup_{j\in J}  X_{i_j} \cap  L$. 
 This proves that $\{X_i\}_{i\in I}$ is admissible. 
\end{proof}

\begin{rem}
Locally small spaces are certain sets with Grothendieck topologies on them.
Proposition \ref{glu} corresponds to Proposition 2 of Subsection 9.1.3 in \cite{BGR}
and Proposition 2.2.42 in \cite{Pie2}.
\end{rem}

\begin{cor} Each locally small space can be written as the admissible union of all its open small subspaces 
$$(X, \mc{L}_X)=\bigcup^a_{L\in\mc{L}_X} (L,\mc{L}_X \cap_1 L).$$
\end{cor}

\subsection{GTS and identification}

\begin{defi}[cf. \cite{DK}, Def.1 in I.1 and \cite{Pie2}, Def. 2.2.2] \label{gts}
A \textbf{(Delfs-Kne\-busch) generalized topological space} (shortly: \textbf{gts}) is a system of the form 
 $(X,Op_X,Cov_X)$, where $X$ is any set, $Op_X\subseteq \mc{P}(X)$,
and $Cov_X\subseteq \mc{P}(Op_X)$,  such that 
the following axioms are satisfied:
\begin{enumerate}{}{\setlength{\leftmargin}{4cm}\setlength{\rightmargin}{.2cm}\setlength{\labelsep}{.2cm}}
\item[(\textbf{finiteness})] if $\mc{U}\subseteq Op_X$ is finite, then  
$\mc{U}\in Cov_X$ and
$\bigcup \mc{U},\bigcap \mc{U}\in Op_X$ (where $\bigcap \emptyset =X$),

\item[(\textbf{stability})]  if $V\in Op_X$ and $\mc{U}\in Cov_X$, then $V\cap_1\mc{U}\in Cov_{X}$, 

\item[(\textbf{transitivity})] if $\mathcal{U}\in\text{Cov}_X$ and, for each $U\in \mathcal{U}$, we have $\mathcal{V}(U)\in\text{Cov}_X$ such that $\bigcup\mathcal{V}(U)=U$, then $\bigcup_{U\in\mathcal{U}}\mathcal{V}(U)\in\text{Cov}_X$,  

\item[(\textbf{saturation})] if $\mc{U}\in Cov_X, \mc{V}\subseteq Op_X, \:\bigcup\mc{U}=\bigcup \mc{V}$ and $\mc{U}$ is a refinement of $\mc{V}$, 
then $\mc{V}\in Cov_X$,

\item[(\textbf{regularity})] if $\mc{U}\in Cov_X$, $W\subseteq \bigcup\mc{U}$, and $W\cap_1\mc{U}\subseteq Op_X$, then $W\in Op_X$.    
\end{enumerate}
Since $Op_X=\bigcup Cov_X$, a generalized topological space is often denoted $(X, Cov_X)$, and $Cov_X$ is called a \textbf{generalized topology}. 
Members of $Cov_X$ are called \textbf{admissible families} and members of $Op_X$ are called \textbf{open sets}.

A mapping between generalized topological spaces is \textbf{strictly continuous} if the preimage of an admissible family is an admissible family. The category of generalized topological spaces and  strictly continuous mappings is denoted by \textbf{GTS}.
\end{defi}

\begin{defi}[cf. \cite{Pie2}, Def. 2.2.25]
A set $S\subseteq X$ is \textbf{small} if each admissible family $\mc{U}$ is essentially finite on $S$ (i. e. $\mc{U} \cap_1 S$ is essentially finite). The family of all small open subsets of $X$ is denoted by $Smop_X$.
\end{defi}

\begin{defi}
A generalized topological space  $(X,Op_X,Cov_X)$  is:
\begin{enumerate}
\item[$(a)$] \textbf{small} if $X$ is a small set (cf. \cite{Pie2}, Def. 2.3.1),
\item[$(b)$] \textbf{locally small} if  there exists an admissible covering of $X$ by small open sets (cf. \cite{Pie3}, Def. 2.1.1).
\end{enumerate}
\end{defi}

\begin{defi}[cf. \cite{PW}, p. 242] For a family of families $\Psi \subseteq \mc{P}^2(X)$, we denote by
$\langle \Psi \rangle_X$  the smallest generalized topology on $X$ containing $\Psi$.
\end{defi}

\begin{defi}[cf. \cite{Pie2}, Def. 2.2.13 and Prop. 2.2.71] For any families $\mc{U}, \mc{V}$ of subsets of $X$, we define
$$EF(\mc{U},\mc{V})= \mbox{the family of subfamilies of $\mc{U}$ essentially finite on members of $\mc{V}$,}$$
$$EssFin(\mc{U})=EF(\mc{U},\{X\}).$$
\end{defi}
The next lemma says that the locally small spaces in the sense of Definition \ref{lss} may be understood as a special kind of Delfs-Knebusch generalized topological spaces.
\begin{lem}\label{id}
For a locally small space $(X, \mc{L}_X)$, all axioms of a Delfs-Knebusch generalized topological space are satisfied by  $(X, \mc{L}^o_X, EF(\mc{L}^o_X,\mc{L}_X))$. 
In particular, the regularity axiom holds:
\begin{enumerate}
\item[(REG)]
if  $\mc{U}$ is an admissible open  family and $W\subseteq \bigcup \mc{U}$ 
 such that $W\cap_1 \mc{U}\subseteq \mc{L}^o_X$, then $W\in \mc{L}^o_X$.
\end{enumerate} 
\end{lem}
\begin{proof}
We are to check if $W\cap L\in \mc{L}_X$ for each $L\in \mc{L}_X$.
Notice that $W\cap L=W\cap L\cap \bigcup \mc{U}=W\cap L\cap (U_1 \cup ...\cup U_k)$ for some $U_1,...,U_k\in \mc{U}$, since $\mc{U}$ is locally essentially finite.
But each $W\cap L\cap U_i$, with $i=1,...,k$, belongs to $\mc{L}_X$, which is a ring of subsets of $X$. That is why $W\cap L \in \mc{L}_X$.  

The rest of the axioms are easy to check.
\end{proof}

\begin{cor}
The union  of an admissible family is an open set.
\end{cor}

\begin{lem}\label{covx}
For a locally small gts $(X, Op_X, Cov_X)$:
\begin{enumerate}
\item[$(a)$] the pair $(X, Smop_X)$ is a locally small space (in the sense of Definition 2.1),
\item[$(b)$] \quad $Cov_X=EF(Smop_X^o, Smop_X)$. 
\end{enumerate}
\end{lem}
\begin{proof}
$(a)$: The family $Smop_X$ contains the empty set, is closed under finite unions and intersections and covers $X$ by the definition of a locally small gts.

$(b)$: We first prove $Op_X=Smop_X^o$.

Obviously, $Op_X\subseteq Smop_X^o$. If $V\in Smop_X^o$, then $V=V\cap \bigcup Smop_X= \bigcup (V\cap_1 Smop_X)$. In the gts $(X,Op_X, Cov_X)$, the family $Smop_X$ is admissible, and $V\cap_1 Smop_X$ is open. By the regularity axiom, the set $V$ is open, so $Smop_X^o  \subseteq   Op_X$.

Now we prove $Cov_X=EF(Op_X,Smop_X)$.

Clearly, $Cov_X \subseteq EF(Op_X,Smop_X)$.  If $\mc{U}  \in EF(Op_X,Smop_X)$, then, for each $V\in Smop_X$, the family $\mc{U}\cap_1 V$ is essentially finite, so admissible. Hence $(\bigcup \mc{U})\cap V\in Smop_X$ and $\bigcup \mc{U}\in Op_X$ by the regularity axiom. Since $Smop_X$ is admissible, also  $(\bigcup \mc{U})\cap_1 Smop_X$ is admissible by the stability axiom, $\mc{U}\cap_1 Smop_X$ is admissible by the transitivity axiom, and $\mc{U}$ is admissible by the saturation axiom.
We have proved that  $EF(Op_X,Smop_X) \subseteq Cov_X$.
\end{proof}

\begin{lem}\label{smops}
If $(X, \mc{L}_X)$ is a locally small space, then the family of  small open sets in the gts
$(X, \mc{L}^o_X, EF(\mc{L}^o_X,\mc{L}_X))$ is $Smop_X=\mc{L}_X$.
\end{lem}
\begin{proof}
Obviously, $\mc{L}_X \subseteq Smop_X$. We prove $Smop_X \subseteq \mc{L}_X$.
Obviously, $Smop_X \subseteq Op_X=\mc{L}_X^o$. If $V$ is small open in $(X,\mc{L}^o_X, EF(\mc{L}^o_X,\mc{L}_X))$, then the family $\mc{L}_X\in EF(\mc{L}^o_X,\mc{L}_X)$ is essentially finite on $V$, so $V\subseteq L_1 \cup ... \cup L_k \in \mc{L}_X$ and
$V\in \mc{L}^s$. Finally, $Smop_X\subseteq \mc{L}_X^s  \cap \mc{L}_X^o = \mc{L}_X$.
\end{proof}

\begin{rem} 
We shall identify a locally small space  $(X, \mc{L}_X)$   with a locally small gts $(X,\mc{L}^o_X, EF(\mc{L}^o_X,\mc{L}_X))$. 
\end{rem}
 
\begin{rem}
Consequently, each small space    $(X, \mc{L}_X)$ is identified with a small
gts $(X,\mc{L}_X,EssFin(\mc{L}_X))$. 
 \end{rem}
\subsection{Mappings between spaces}

\begin{defi}
Assume $(X, \mc{L}_X)$ and $(Y,\mc{L}_Y)$  are locally small spaces.
A mapping $f:X \to Y$ will be called: 

\begin{enumerate}
\item[(a)] \textbf{weakly continuous}  if 
$f^{-1}(\mc{L}^{wo}_Y)\subseteq \mc{L}^{wo}_X$,
\item[(b)]  \textbf{bounded} if $\mc{L}_X$ is a refinement of $f^{-1}(\mc{L}_Y)$,
\item[(c)]  \textbf{continuous} if $ f^{-1}(\mc{L}_Y) \subseteq  \mc{L}^o_X$ (i. e. $f^{-1}(\mc{L}_Y)$ is compatible with $\mc{L}_X$). 
\end{enumerate} 
\end{defi}

\begin{prop} Each continuous mapping is weakly continuous.
\end{prop}
\begin{proof}
Follows from the fact that the preimage of a union equals the union of preimages.
\end{proof}

\begin{prop}
For locally small spaces  $(X, \mc{L}_X)$, $(Y,\mc{L}_Y)$, 
a mapping $f:X \to Y$ is bounded continuous iff it satisfies \\
\ \quad $(bc)$ $f(\mc{L}_X^s)\subseteq \mc{L}_Y^s$ and  $f^{-1}(\mc{L}^{o}_Y)\subseteq \mc{L}^{o}_X$.
\end{prop}

\begin{proof} Obviously, $(bc)$ implies $(b)$ and $(c)$. We need to prove that $f$ bounded continuous implies $(bc)$.
Assume $V \in  \mc{L}^o_Y$ and $W \in \mc{L}_X$. We need to prove $f^{-1}(V)\cap W \in \mc{L}_X$. Choose some $U \in \mc{L}_Y$ containing $f(W)$.
Then
$$f^{-1}(V)\cap W=f^{-1}(V\cap U)\cap W 
\in f^{-1} (\mc{L}_Y) \cap_1 \mc{L}_X \subseteq \mc{L}_X. $$
\end{proof}

\begin{exam}
Without the boundedness assumption, the equivalence above does not hold. Consider the mapping $f=id_{\mb{R}}:\mb{R}_{om} \to \mb{R}_{lom}$. Then $f^{-1}(\mc{L}_{lom})\subsetneq \mc{L}^{o}_{om}$, but $ f^{-1}(\mc{L}^{o}_{lom})$ contains much more sets than $\mc{L}^{o}_{om}$.
\end{exam}

\begin{lem}  \label{bcsc}
For locally small spaces  $(X, \mc{L}_X)$, $(Y,\mc{L}_Y)$ and a mapping $f:X\to Y$, the following conditions are equivalent:
\begin{enumerate}
\item[$(1)$] $f$ is bounded continuous,
\item[$(2)$] $f$ is strictly continuous (i. e. $f^{-1}(EF(\mc{L}^o_Y,\mc{L}_Y))
\subseteq EF(\mc{L}^o_X,\mc{L}_X)$).
\end{enumerate}
\end{lem}
\begin{proof}
$(1) \Rightarrow (2)$ Let $\mc{V}\in EF(\mc{L}^o_Y,\mc{L}_Y))$. 
Since $f^{-1}(\mc{L}^o_Y)\subseteq \mc{L}^o_X$ and $\mc{L}_X$ is a refinement of $f^{-1}(\mc{L}_Y)$, the family $f^{-1}(\mc{V})$ is open and essentially finite on each member of $f^{-1}(\mc{L}_Y)$, so also on each member of $\mc{L}_X$. Hence
$f^{-1}(\mc{V})\in EF(\mc{L}^o_X,\mc{L}_X)$. The strict continuity of $f$ is proved.

$(2) \Rightarrow (1)$ 
Each strictly continuous mapping is continuous. 
By Proposition 2.1.22 of \cite{Pie3}, Proposition 2.2.26 of \cite{Pie2} it is also bounded.
\end{proof}

We give a strong restatement of Theorem 2.1.33 of \cite{Pie3} now.
\begin{thm} \label{1ki}
The category of locally small gtses with strictly continuous mappings is concretely isomorphic to the category of locally small spaces and continuous bounded mappings.
\end{thm}
\begin{proof}
Follows from Lemmas  \ref{id},  \ref{covx}, \ref{smops} and \ref{bcsc}.
\end{proof}

\begin{rem}
The category of locally small spaces and their bounded continuous mappings will be denoted by  \textbf{LSS}.
This category was denoted by \textbf{Sublat} in \cite{Pie3}. Since  \textbf{Sublat}  and \textbf{LSS}  are concretely isomorphic constructs (compare \cite{AHS}, Remark 5.12), we identify them. 
\end{rem}

\begin{cor}
For small spaces $(X, \mc{L}_X)$, $(Y,\mc{L}_Y)$, the following conditions are equivalent for a mapping $f:X \to Y$:  
\begin{enumerate}
\item[$(a)$] $f$ is bounded continuous, 
\item[$(b)$] $f$ is continuous, 
\item[$(c)$] $f^{-1}(\mc{L}_Y)  \subseteq \mc{L}_X$.
\end{enumerate}
\end{cor}

\begin{rem}
The category of  small spaces and their  continuous mappings, concretely isomorphic to the category of small gtses and their (strictly) continuous mappings,  will be denoted by  \textbf{SS},
in accordance with \cite{Pie2, PW}.
\end{rem}

\subsection{Arbitrary subspaces}

\begin{defi}
A  \textbf{subspace}  of $(X, \mc{L}_X)$  induced by a subset $Y\subseteq X$ 
is the pair $(Y, \mc{L}_X \cap_1 Y)$.
\end{defi}

\begin{thm} \label{subsp}
The subspace induced by $Y\subseteq X$ in the sense of the above definition is, after identification, the same subspace that is induced by $Y$ in  $(X,\mc{L}^o_X, EF(\mc{L}^o_X,\mc{L}_X))$ as an object of $\mathbf{GTS}$.
\end{thm}

\begin{lem} $EF(\mc{L}^o_X,\mc{L}_X)=\langle \{ \mc{L}_X\} \rangle_X$.
\end{lem}
\begin{proof}
Obviously, $\mc{L}_X$ belongs to  the generalized topology $EF(\mc{L}^o_X,\mc{L}_X)$, so $\supseteq $ holds.

We prove $EF(\mc{L}^o_X,\mc{L}_X)\subseteq \langle \{ \mc{L}_X\} \rangle_X$.\\
1) Notice that $\mc{L}_X \subseteq \bigcup \langle \{ \mc{L}_X\} \rangle_X$. 
By Proposition 2.2.23 of \cite{Pie2},  $EssFin(\mc{L}_X)\subseteq  \langle \{ \mc{L}_X\} \rangle_X$. 
In particular,  for each $L \in \mc{L}_X$ we have $L\cap_1 \mc{L}_X \in \langle \{ \mc{L}_X\} \rangle_X$.\\
2) Each $V\in \mc{L}^o_X$ is the union of the family $V\cap_1 \mc{L}_X\subseteq \bigcup \langle \{ \mc{L}_X\} \rangle_X$.
By the regularity axiom applied to $V$ and $\mc{L}_X$, we have $V\in  \bigcup \langle \{ \mc{L}_X\} \rangle_X$.
By the stability axiom,  we have $V\cap_1 \mc{L}_X\in \langle \{ \mc{L}_X\} \rangle_X$; moreover, we get
$EssFin(\mc{L}^o_X)\subseteq  \langle \{ \mc{L}_X\} \rangle_X$. \\
3) Take any $\mc{U}\in EF(\mc{L}^o_X,\mc{L}_X)$. Then  for each $L\in \mc{L}_X$ 
we have  $L  \cap_1  \mc{U} \in EssFin(\mc{L}_X)$. Since $\bigcup \mc{U}\in \mc{L}^o_X$ and $\mc{L}_X \cap_1 (\bigcup \mc{U})\in \langle \{ \mc{L}_X\} \rangle_X$,
we get  $\mc{L}_X\cap_1 \mc{U}\in \langle \{ \mc{L}_X\} \rangle_X$ by the transitivity axiom, and  $\mc{U}\in \langle \{ \mc{L}_X\} \rangle_X$ by the saturation axiom.
\end{proof}

\begin{proof}[Proof of the Theorem \ref{subsp}]
By Definition 2.2.41 of \cite{Pie2}, we are to check that 
$$\langle  \langle \{ \mc{L}_X\} \rangle_X \cap_2 Y \rangle_Y=\langle \{ \mc{L}_X \cap_1 Y \} \rangle_Y.$$

Obviously, $ \{ \mc{L}_X \cap_1 Y \} =  \{ \mc{L}_X\}  \cap_2   Y  \subseteq 
 \langle \{ \mc{L}_X\} \rangle_X \cap_2 Y $, which gives $\supseteq$.
 
 By Proposition  2.2.37 of \cite{Pie2}, $\langle \{ \mc{L}_X\} \rangle_X \cap_2 Y \subseteq \langle \{ \mc{L}_X \cap_1 Y \} \rangle_Y$.
\end{proof}


%
\subsection{Paracompact, {Lindel\"of} and regular spaces}

\begin{prop}[cf. \cite{DK}, Prop. I.4.5]
For a locally small space $(X, \mc{L}_X)$, the following conditions are equivalent:
\begin{enumerate}
\item[$(a)$]  any admissible covering of $X$ by smops admits a refinement that is a locally finite covering of $X$ by smops,
\item[$(b)$]  there exists a locally finite subfamily of $\mc{L}_X$ covering~$X$.
\end{enumerate}
\end{prop}
\begin{proof}
$(a) \Rightarrow (b)$: The family $\mc{L}_X$ is an admissible covering of $X$ by smops, so it admits a locally finite refinement by smops, which is a subcovering of  $\mc{L}_X$.

$(b) \Rightarrow (a)$:
Let $\mc{K} \subseteq \mc{L}_X$ be a locally finite subcovering of $X$ and let $\mc{L}$ be an admissible covering of $X$ by smops.
The family  $\mc{L}$ is essentially finite on elements of $\mc{K}$, so for each $K \in \mc{K}$ there exists a finite subfamily  $\mc{L}_K\subseteq \mc{L}$ such that 
$K \subseteq \bigcup \mc{L}_K$. Consider  
$\tilde{\mc{L}}=\bigcup_{K\in \mc{K}} (\mc{L}_K  \cap K)$. This family of smops is a locally finite refinement of $\mc{L}$ and covers $X$.
\end{proof}

\begin{defi}[cf. \cite{DK}, Def. 2 in I.4]
A locally small space $(X, \mc{L}_X)$  is called \textbf{paracompact} if
it satisfies one of the equivalent conditions in the previous proposition.
\end{defi}
\begin{prop}[cf. \cite{DK}, Prop. I.4.16]
For a locally small space $(X, \mc{L}_X)$, the following conditions are equivalent:
\begin{enumerate}
\item[$(a)$] each admissible covering of $X$ by smops admits a countable admissible subcovering,
\item[$(b)$]  there exists a countable admissible subfamily  of $\mc{L}_X$ covering~$X$.
\end{enumerate}
\end{prop}
\begin{proof}
$(a) \Rightarrow (b)$: 
The admissible covering  $\mc{L}_X$ admits a countable  admissible subcovering.

$(b) \Rightarrow (a)$:
Let $\mc{C} \subseteq \mc{L}_X$ be an admissible countable subcovering of $X$ and let $\mc{L}$ be an admissible covering of $X$ by smops.
The family  $\mc{L}$ is essentially finite on elements of $\mc{C}$, so for each $C \in \mc{C}$ there exists a finite subfamily  $\mc{L}_C\subseteq \mc{L}$ such that 
$C \subseteq \bigcup \mc{L}_C$. 
The family $\mc{N}=\bigcup_{C\in \mc{C}}   \mc{L}_C$  is a countable subcovering of $\mc{L}$ and is admissible, since it is essentially finite on elements of 
$\mc{C}$ and $\mc{C}$ is essentially finite on elements of  $\mc{L}_X$.
\end{proof}

\begin{defi}[cf. \cite{DK}, Def. 3 in I.4]
A locally small space $(X, \mc{L}_X)$  is  called \textbf{Lindel\"of} if
it satisfies one of the equivalent conditions in the previous proposition.
\end{defi}
\begin{defi}[cf. \cite{DK}, I.3, p. 35]   A locally small space $(X, \mc{L}_X)$  is called  
 \textbf{connected} if $X$ cannot be written as a disjoint union  of its two non-empty open subsets. 
\end{defi}

\begin{thm}[cf. \cite{DK}, Thm. I.4.17] \label{cpar}
Any connected paracompact locally small space is Lindel\"of.
\end{thm}
\begin{proof}
Let $\mc{K}$ be a locally finite subfamily of $\mc{L}_X$ covering $X$. 
Choose $M\in \mc{K}\setminus \{ \emptyset \}$ and  define 
$M_0=M$, $M_{n+1}=\bigcup \{K\in \mc{K}: K\cap M_n\neq \emptyset \}$ for 
$n\in \mb{N}$ and $M_{\infty}=\bigcup_{n\in \mb{N}} M_n$.
Define $X_1=M_{\infty}, X_2=X \setminus M_{\infty}$. Both $X_1,X_2$ are open.
Indeed, any $L\in \mc{L}_X$ is contained in some finite union of elements of $\mc{K}$.
But for $K\in \mc{K}$ either $K \cap X_i=\emptyset$ or $K \subseteq X_i$  $(i=1,2)$.
Hence $L\cap X_i$ is a finite union of elements of  $\mc{L}_X$, so an element of 
$\mc{L}_X$.

Since the space is connected, $M_{\infty}=X$. Since $M_{\infty}$ is a countable union of members of $\mc{K}$, the space is Lindel\"of.
\end{proof}

\begin{defi}
The family of \textbf{closed sets} in a locally small space is the family
$$\mc{L}_X^c=X\setminus_1 \mc{L}_X^o= 
\{ X \setminus V \: |\:  V\in \mc{L}_X^o \}.$$

A locally small space is \textbf{(strongly) regular} if: \\
a) each singleton is closed,\\
b) for each $x\in X$ and each $F\in \mc{L}_X^c$
not containing $x$, there exist $U,V \in \mc{L}_X^o$ such that $x\in U$, $F \subseteq V$ and $U\cap V = \emptyset$.
\end{defi}

\begin{defi}[cf. \cite{DK}, Def. 2 in I.7]
A locally small space $(X,\mc{L}_X)$ is:
\begin{enumerate}
\item[a)]  \textbf{taut} if the weak closure of each member of 
$\mc{L}_X$ is small (so contained in another member of $\mc{L}_X$), 
\item[b)] \textbf{strongly taut} if the weak closure  of each member of 
$\mc{L}_X$ is small and closed.
\end{enumerate}
\end{defi}

\begin{thm}[cf. \cite{DK}, Prop. I.4.18] \label{stlind}
Each strongly taut, Lindel\"of space is paracompact.
\end{thm}
\begin{proof}
Take an admissible  countable  family $\{L_n \}_{n\in \mb{N}}\subseteq \mc{L}_X$ covering $X$. We may assume that this family is increasing.
 By tautness, for each $n\in \mb{N}$ there exists $m\in \mb{N}$ such that  
$wcl(L_n) \subseteq L_m$. By choosing a suitable subfamily, we get an admissible countable increasing covering $\{M_n \}_{n\in \mb{N}}$ of $X$ with property  $wcl(M_n) \subseteq M_{n+1}$.

Finally, we define a locally finite covering 
$$W_1=M_1, W_2=M_2, W_{n+2}=M_{n+2} \setminus 
wcl(M_n), \quad n\in \mb{N}$$
 of $X$. By strong tautness, this covering  consists only of smops.
\end{proof}

\subsection{Open bornological universes.}

\begin{lem}  \label{258}
For any locally small space $(X, \mc{L}_X)$, we have $(\mc{L}_X^{swo})^o=\mc{L}_X^{wo}$.
\end{lem}
\begin{proof}
We have $(\mc{L}_X^{swo})^o \subseteq (\mc{L}_X^{swo})^{wo} \subseteq
(\mc{L}_X^{wo})^{wo} =\mc{L}_X^{wo}$.

On the other hand, if $W\in \mc{L}_X^{wo}$, then for $Z\in \mc{L}_X^{swo}$ we have
$W\cap Z \in \mc{L}_X^{swo}$. This means $W\in (\mc{L}_X^{swo})^o$.
\end{proof}

\begin{prop}
For a locally small space $(X, \mc{L}_X)$, the following are equivalent:
\begin{enumerate}
\item[$(1)$]   \quad  $\mc{L}_X=\mc{L}_X^{swo}$,
\item[$(2)$]  \quad $\mc{L}^o_X=\mc{L}_X^{wo}$.
\end{enumerate}
\end{prop}
\begin{proof}
$(1) \Rightarrow (2)$ 
By Lemma \ref{258}, we have $\mc{L}^o_X=    (\mc{L}_X^{swo})^o   =\mc{L}_X^{wo}$. \\
$(2) \Rightarrow (1)$ 
We have $\mc{L}_X= \mc{L}_X^s \cap \mc{L}_X^o= \mc{L}_X^s \cap \mc{L}_X^{wo}= \mc{L}_X^{swo}$. 
\end{proof}

\begin{defi}[cf. \cite{Pie2},  Prop. 2.2.71 and   \cite{Pie3}, Prop. 2.1.31]
A locally small space $(X,\mc{L}_X)$ is called \textbf{partially topological} if
it satisfies one of the equivalent conditions in the previous proposition.
The full subcategory in \textbf{LSS} of partially topological locally small spaces will be denoted by $\mathbf{LSS}_{pt}$.
\end{defi}

\begin{defi}[cf. \cite{Pie3}, Def. 2.2.29]
A \textbf{bornological universe} is a triple $(X, \tau_X, \mc{B}_X)$.
A \textbf{basis} of bornology $\mc{B}_X$ is a family $\mb{B} \subseteq \mc{B}_X$ such that every element of  $\mc{B}_X$ is contained in some element of $\mb{B}$.
If all elements of  $\mb{B}$ are open in $(X,\tau_X)$, then we call it an \textbf{open basis}.

Assume that another bornological universe $(Y, \tau_Y, \mc{B}_Y)$ is given.
Then a mapping $f: X\to Y$ is called \textbf{bounded} if 
$f(\mc{B}_X)\subseteq  \mc{B}_Y$.
\end{defi}

\begin{rem}
The category of bornological universes having open bases of the bornology (shortly: open bornological universes) and their bounded continuous mappings  will be denoted by  \textbf{OpenBorUniv}, in accordance with \cite{SV}. This category  was 
denoted \textbf{UBorOB} in  \cite{Pie3}.
\end{rem}
We give a strong restatement of Proposition 2.1.31 in \cite{Pie3}.
\begin{thm} \label{2ki}
The constructs $\mathbf{LSS}_{pt}$ and $\mathbf{OpenBorUniv}$ are concretely isomorphic.
\end{thm}
\begin{proof}
We have a concrete functor $ubor(X, \mc{L}_X)=(X, \mc{L}_X^{wo},  \mc{L}_X^s)$
from $\mathbf{LSS}_{pt}$ to $\mathbf{OpenBorUniv}$ 
 and a concrete functor $lss(X, \tau,\mc{B})=(X, \tau \cap \mc{B})$
from $\mathbf{OpenBorUniv}$ to $\mathbf{LSS}_{pt}$.

1. In both categories the morphisms are the bounded continuous mappings (with the same meaning of ``bounded continuous'').

2. The functor $lss \circ ubor$ is the identity on $\mathbf{LSS}_{pt}$, since
$\mc{L}_X= \mc{L}_X^s \cap \mc{L}_X^{wo}$.

3. The functor $ubor \circ lss$ is the identity on $\mathbf{OpenBorUniv}$.
 Indeed, since $\tau \cap \mc{B}$ is an open basis of the bornology $\mc{B}$, we get
$\bor(\tau \cap \mc{B})=\mc{B}$. Obviously $\topo(\tau \cap \mc{B})\subseteq \tau $.
If $U \in \tau$, then for each $u\in U$ there exists some $V_u\in \tau\cap\mc{B}$ such that $u\in V_u$. So $U=\bigcup_{u \in U}  (V_u \cap U) \in \topo(\tau \cap \mc{B})$. 
That is why $\tau \subseteq \topo(\tau \cap \mc{B})$. Hence 
$ubor \circ lss(X,\tau, \mc{B})=(X,\tau, \mc{B})$.  
\end{proof}

\subsection{Topological-like spaces.}

\begin{defi}
A locally small space $(X, \mc{L}_X)$ will be called \textbf{small partially topological} 
(or  \textbf{topological-like}) if $\mc{L}_X$ is a topology (equivalently: $\mc{L}_X = \mc{L}^{wo}_X$).
The full subcategory in \textbf{LSS} of small partially topological spaces will be denoted by $\mathbf{SS}_{pt}$ (it is concretely isomorphic to the category of small partially topological gtses and their continuous mappings).
\end{defi}

\begin{prop}[cf. \cite{Pie3}, Prop. 2.3.18]
The traditional category $\mathbf{Top}$ of topological spaces and their continuous mappings is concretely isomorphic to the category $\mathbf{SS}_{pt}$ of partially topological small spaces and their (strictly) continuous mappings.  
\end{prop}
\begin{proof}
The correspondence 
$$(X,\tau) \to (X,\tau, EssFin(\tau))$$
gives a concrete isomorphism of constructs.
\end{proof}

\begin{rem}[cf. \cite{Pie2}, Rem. 2.2.63]
A more natural embedding of \textbf{Top} into \textbf{GTS} is given by the correspondence
$$(X,\tau) \to (X,\tau, \mc{P}(\tau)).$$
Most often the gts $(X,\tau, \mc{P}(\tau))$ is not  locally small.
\end{rem}
\subsection{A concrete reflector  and a concrete  coreflector}

\begin{thm}\label{refl}
$(1)$ The functor of smallification  $sm:\mathbf{LSS} \to \mathbf{SS}$ 
defined by $sm(X, \mc{L}_X)=(X, \mc{L}^o_X)$
is a concrete reflector.\\
$(2)$  After identification, it is the restriction of the functor of smallification
$sm:\mathbf{GTS} \to \mathbf{SS}$ considered in Proposition 2.3.16 of \cite{Pie2}. 
\end{thm}
\begin{proof}
$(1)$  The reflection for $(X, \mc{L}_X)$ is the mapping
$$ r_X = id_X: (X, \mc{L}_X)   \to   (X, \mc{L}^o_X). $$
Indeed, for any morphism $f:(X, \mc{L}_X) \to (Y, \mc{L}_Y)$ in \textbf{LSS} into a small space $(Y, \mc{L}_Y)$ there exists a unique morphism $\hat{f}:(X, \mc{L}^o_X) \to (Y, \mc{L}_Y)$ such that $f=\hat{f} \circ r_X$, where $\hat{f}=f$ as functions.
By Proposition 4.22 of \cite{AHS}, all the (identity-carried) reflections form a functor that is a concrete reflector.

$(2)$ The functor $sm:\mathbf{GTS} \to \mathbf{SS}$ is defined  by 
$$ sm(X, Op_X, Cov_X)=(X, Op_X, EssFin(Op_X)).$$
In our situation, the object
$(X,\mc{L}^o_X)$ is identified with the gts 
$$(X,\mc{L}^{oo}_X,EF(\mc{L}^{oo}_X,\mc{L}^o_X))=(X,\mc{L}^o_X,EssFin(\mc{L}^{o}_X)).$$
\end{proof}

\begin{thm} \label{korefl}
$(1)$ The functor of partial topologization   $pt:\mathbf{LSS} \to \mathbf{LSS}_{pt}$
defined by $pt(X, \mc{L}_X)=(X, \mc{L}^{swo}_X)$
 is a concrete coreflector.\\
$(2)$ After identification, it is  the restriction of the functor $pt:\mathbf{GTS} \to \mathbf{GTS}_{pt}$ considered in Definition 4.1 of \cite{PW}.
\end{thm}
\begin{proof}
$(1)$
The coreflection for  $(X, \mc{L}_X)$  is the mapping 
$$  c_X = id_X: (X, \mc{L}^{swo}_X)   \to   (X, \mc{L}_X). $$
Indeed,  for any morphism $f:(Y, \mc{L}_Y) \to (X, \mc{L}_X)$ in \textbf{LSS} from a partially topological space $(Y, \mc{L}_Y)$ there exists a unique morphism 
$\hat{f}:(Y, \mc{L}_Y) \to (X, \mc{L}^{swo}_X)$ such that $f=c_X \circ \hat{f}$, where $\hat{f}=f$ as functions. That $\hat{f}$ is a morphism follows from the equality
$\mc{L}_X^s=(\mc{L}_X^{swo})^s$ and Lemma \ref{258}.
By Proposition 4.27 of \cite{AHS}, all the (identity-carried) coreflections form a functor that is a concrete coreflector.

$(2)$  
The functor $pt:\mathbf{GTS}\to \mathbf{GTS}_{pt}$ is defined by 
$$pt(X,Cov_X)=(X,\langle Cov_X \cup EssFin(\topo(Op_X))\rangle_X ). $$
We are to check if $EF(\mc{L}^{wo}_X,\mc{L}^{swo}_X)$ is equal to 
$\langle EF(\mc{L}^{o}_X,\mc{L}_X) \cup EssFin(\mc{L}^{wo}_X) \rangle_X$,
the latter family of families will be denoted by $\Psi$.
Clearly, 
$$EF(\mc{L}^{o}_X,\mc{L}_X) \subseteq EF(\mc{L}^{wo}_X,\mc{L}^{swo}_X)
\mbox{ and }EssFin(\mc{L}^{wo}_X) \subseteq  EF(\mc{L}^{wo}_X,\mc{L}^{swo}_X).$$
It remains to check that $EF(\mc{L}^{wo}_X,\mc{L}^{swo}_X)\subseteq \langle EF(\mc{L}^{o}_X,\mc{L}_X) \cup EssFin(\mc{L}^{wo}_X) \rangle_X$.

Let $\mc{U}\in EF(\mc{L}^{wo}_X,\mc{L}^{swo}_X)$. For each $L \in \mc{L}_X$, the family $\mc{U} \cap_1  L$ is weakly open and essentially finite, so $\mc{U}\in EssFin(\mc{L}^{wo}_X)$. Now $(\bigcup \mc{U})\cap_1 \mc{L}_X$ belongs to $\Psi$ by the stability axiom,  $\mc{U}\cap_1 \mc{L}_X \in \Psi$ by the transitivity axiom   and 
$\mc{U} \in \Psi$ by the saturation axiom.
\end{proof}

\begin{rem}
There exist many important functors in mathematics that are  reflectors, but not concrete reflectors (see \cite{AHS}, Examples 4.17 (3) --- (12)). Similarly, there exist many coreflectors that are not concrete coreflectors (see \cite{AHS}, Examples 4.26 (3) and (4)). 
This emphasizes the importance of functors $sm$ and $pt$ above. Notice that the compositions $sm \circ pt$ and $pt \circ sm$ are equal on \textbf{GTS}.
\end{rem}

\section{Building locally definable spaces} 

\subsection{Spaces over sets}   \ \ \ 

\vspace{2mm}
We re-establish the theory of locally definable spaces from \cite{Pie3}, Section 3.

Assume $M$ is any non-empty set. 

\begin{defi}
A \textbf{function sheaf over $M$} on a locally small space  $\mc{X}=(X, \mc{L}_X)$ is a family 
$\mc{F}$ of functions $h:U\to M$, where $U\in \mc{L}^o_X$,
which is closed under:
\begin{enumerate}
\item[a)] restrictions of functions to open subsets $V\subseteq U$,

\item[b)] gluings of compatible families of functions defined on members of  locally essentially finite open families.\\
\end{enumerate}
  
Denote by $FSh(\mc{X},M)$ the family of all function sheaves on a space $\mc{X}$
over a set~$M$.
\end{defi}

\begin{defi}
For function sheaves $\mc{F}\in FSh(\mc{X},M)$, 
$\mc{G}\in FSh(\mc{Y},M)$
and a  mapping $f:X\to Y$, we define the following families
$$f_{*}\mc{F}=\{ h:V\to M |V\in \mc{L}^o_Y,\: h\circ f\in \mc{F}\},  \mbox{ called the \textbf{image} of }\mc{F}\mbox{ by }f, $$
$$ f^{-1}\mc{G}=\mc{G}\circ f=\{h\circ f|h\in \mc{G}\},   \mbox{ called the \textbf{preimage} of   } \mc{G} \mbox{ by }  f.$$
\end{defi}

\begin{defi}
A \textbf{locally small space over $M$} is a pair $(\mc{X}, \mc{O}_X )$, where 
$\mc{X}$ is a locally small space  and $\mc{O}_X$ is  a function sheaf over $M$ on $X$.
A \textbf{morphism} $f:(\mc{X}, \mc{O}_X )\rightarrow (\mc{Y}, \mc{O}_Y)$ \textbf{of locally small spaces  over} $M$ is a  morphism $f: \mc{X}\to \mc{Y}$ in \textbf{LSS} (i. e. a strictly continuous mapping $f: X\rightarrow Y$) such that $\mc{O}_Y\subseteq f_*\mc{O}_X$ (equivalently: $f^{-1}\mc{O}_Y\subseteq \mc{O}_X$).

We get the category $\lss(M)$ of locally small  spaces over $M$ and their morphisms.
\end{defi}

\begin{defi}
For an object $((X,\mc{L}_X),\mc{O}_X)$ of $\lss(M)$ and an open subset $Y\subseteq X$ (i. e. $Y \in \mc{L}^o_X$), we induce:  
  $$\mc{L}_Y= \mc{L}_X \cap_1 Y, \quad \mc{O}_Y= i_{YX}^{-1}\mc{O}_X,$$
  where $i_{YX}: Y \to X$ is the inclusion.
Then $((Y,\mc{L}_Y),\mc{O}_Y)$ will be called an \textbf{open subspace} of  $((X,\mc{L}_X),\mc{O}_X)$   in $\lss(M)$.
\end{defi} 

\begin{defi}
An object $(\mc{X}, \mc{O}_X)$ of $\lss(M)$ is
the \textbf{admissible union} of a family $\{(\mc{X}_i, \mc{O}_i)\}_{i\in I}$  of  locally small spaces  over $M$, 
if  $\displaystyle (X, \mc{L}_X)= \bigcup^a_{i\in I} (X_i, \mc{L}_i)$ and $\mc{O}_X$ is the smallest  function sheaf containing $\bigcup_{i\in I} \mc{O}_i$. We write then
$$(\mc{X}, \mc{O}_X)= \bigcup^a_{i\in I} (\mc{X}_i, \mc{O}_i). $$
If $I$ is finite, such an admissible union will be called a   \textbf{finite open union} of objects of $\lss(M)$.
\end{defi}
\begin{defi}
 An object  of $\lss(M)$  is \textbf{small}, \textbf{connected}, \textbf{regular}, \textbf{paracompact} or \textbf{Lindel\"of} if its underlying locally small space is such.
\end{defi}

\subsection{Structures with topologies}

\begin{defi}[cf. \cite{Pie3}, Def. 3.2.1 and Rem. 3.2.2]
 A \textbf{structure with a topology} (or a \textbf{weakly topological structure}) is a 
pair $(\mc{M},\sigma)$ composed of a (first order, one-sorted) structure $\mc{M}=(M,...)$ in the sense of model theory and a topology $\sigma$  on the underlying set $M$ of $\mc{M}$. This gives the product topologies $\sigma^n$ on Cartesian powers $M^n$ and the induced topologies $\sigma_D$ on definable (with parameters) sets $D\subseteq M^n$.
For a \textbf{topological structure} one assumes that some basis of $\sigma$ is an 
$\mc{M}$-definable family of subsets of $M$.
\end{defi}

\begin{exam}
O-minimal structures (studied in \cite{vdd}) are examples of topological structures.
If $\mc{M}=(M,< ,...)$ is an o-minimal structure with $<$ a binary relation (interdefinable with) a linear order, then
the formula $a<x<b$ gives a definable family of open intervals, which is a basis of the order topology on $M$.
\end{exam}

From now on, fix a  structure with a topology $(\mc{M},\sigma)$, no connection between $\sigma$ and primitive relations and functions of $\mc{M}$ is assumed. 

\begin{defi}[cf. \cite{Pie3}, Def. 3.2.7]  \label{ad}
For each definable (with parameters) set $D\subseteq M^n$, 
 we  set:
 \begin{enumerate}
 \item[(i)] the \textbf{family of smops} $\mc{L}_D$ of $D$ is the family of  $\sigma_D$-open, $\mc{M}$-definable subsets of $D$,
 \item[(ii)] the \textbf{structure sheaf} $\mc{DC}_D$ of $D$ is the function sheaf of all $\mc{M}$-definable  $(\sigma_D,\sigma)$-continuous functions 
from respective  $U\in  \mc{L}_D$  into $M$.
\end{enumerate}
Hence $((D,\mc{L}_D), \mc{DC}_D)$ becomes a small  object of $\lss(M)$. 
\end{defi}

\begin{fact}\label{proj}
For each definable $D\subseteq M^n$, all the projections $\pi_i:D\to M$ ($i=1,...,n$) have the following properties:
\begin{enumerate}
\item[$(1)$]  are $\mc{M}$-definable $(\sigma_D, \sigma)$-continuous,
\item[$(2)$]  belong to $\mc{DC}_D$,
\item[$(3)$]  are morphisms of   $\lss(M)$. 
\end{enumerate}
\end{fact}

\begin{prop}[cf. \cite{Pie3}, Prop. 3.2.13]\label{morph}
For a mapping $f=(f_1,...,f_n):E\to D$ with definable $E\subseteq M^m$ and 
$D\subseteq M^n$, the following conditions are equivalent:
\begin{enumerate}
\item[$(1)$] $f$ is a morphism of $\lss(M)$,
\item[$(2)$] $f$ is $\mc{M}$-definable and $(\sigma_E,\sigma_D)$-continuous,
\item[$(3)$] all of the coordinates of $f$ are functions from $\mc{DC}_E$.
\end{enumerate}
\end{prop}
\begin{proof}
$(1) \Rightarrow (3)$
We have $f=(\pi_1 \circ f, ..., \pi_n \circ f)$. By Fact \ref{proj}, all 
$\pi_1,...,\pi_n:D\to M$ are morphisms of $\lss(M)$,
so all the coordinates $\pi_i \circ f \: (i=1,...,n)$ of $f$ are morphisms of $\lss(M)$. 
But $id_M \in \mc{DC}_M$.  Hence
all $\pi_i \circ f: E \to M$  $(i=1,...,n)$ belong to $\mc{DC}_{E}$.

$(3)\Rightarrow (2)$
Since all the coordinates of $f$ are $\mc{M}$-definable and $(\sigma_E,\sigma)$-continuous, the mapping $f$ is $\mc{M}$-definable and $(\sigma_E,\sigma_D)$-continuous. 

$(2) \Rightarrow (1)$
The mapping $f$ is clearly strictly continuous. Assume $h\in \mc{DC}_D$.
We check if $h\circ f\in \mc{DC}_E$. Since $f$ is $\mc{M}$-definable and $(\sigma_E,\sigma_D)$-continuous and  $h$ is $\mc{M}$-definable and $(\sigma_D,\sigma)$-continuous, $h\circ f$ is $\mc{M}$-definable and $(\sigma_E,\sigma)$-continuous and we are done.
\end{proof}

\begin{cor}\label{homeo}
For two definable sets $D\subseteq M^m$ and $E\subseteq M^n$, the following conditions are equivalent:
\begin{enumerate}
\item  $((D,\mc{L}_D), \mc{DC}_D)$ and  $((E,\mc{L}_E), \mc{DC}_E)$
are isomorphic as objects of $\lss(M)$, 
\item $D$ and $E$ are definably homeomorphic.
\end{enumerate}
\end{cor}

\subsection{Locally definable spaces over  structures with topologies}

\begin{defi}[cf. \cite{Pie3}, Def. 3.3.1, 3.3.2, 3.4.1]
An \textbf{affine definable space over} $(\mc{M}, \sigma )$ is an object of $\lss(M)$ isomorphic to  $((D,\mc{L}_D), \mc{DC}_{D})$  for some definable subset  $D\subseteq M^n$.

A \textbf{locally definable space  over} $(\mc{M},\sigma)$ is an object of $\lss(M)$ that 
is an admissible union of some  affine definable  spaces. 
In particular, a finite open union of some affine definable spaces is called 
a \textbf{definable space over} ($\mc{M},\sigma )$.

\textbf{Morphisms of  (affine or locally) definable spaces  over} $(\mc{M},\sigma )$ are their morphisms in $\lss(M)$. We obtain full subcategories $\ads(\mc{M},\sigma)$, $\lds(\mc{M}, \sigma)$ and $\ds(\mc{M},\sigma)$  of $\lss(M)$.
\end{defi}

\begin{rem}
Definable spaces  over o-minimal structures were extensively studied by L. van den Dries (\cite{vdd}), locally semialgebraic spaces by  H. Delfs and M. Knebusch (\cite{DK}), and locally definable spaces over o-minimal expansions of fields by A. Pi\k{e}kosz (\cite{Pie1}). Especially regular paracompact locally definable spaces are interesting, because the o-minimal homotopy theory is available for them.
\end{rem}

\begin{defi}
A \textbf{locally definable set} in a locally definable space $((X,\mc{L}_X), \mc{O}_X)= \displaystyle \bigcup^a_{i\in I} ((X_i,\mc{L}_i), \mc{O}_i)$ is a subset $Y\subseteq X$  having a definable trace on each $X_i$. A \textbf{definable set} is a small locally definable set.
\end{defi}

\begin{fact}
By gluing the corresponding topological spaces $(X_i,\sigma_{X_i})$, each locally definable space $((X, \mc{L}_X), \mc{O}_X)$ has a natural topology $\sigma_X$. Then:
\begin{enumerate} 
\item[$(1)$]  $\mc{L}_X$ is the family of  all definable $\sigma_X$-open subsets of $X$, 
\item[$(2)$]  $\mc{L}^o_X$  is the family of  all locally definable $\sigma_X$-open subsets of $X$,
\item[$(3)$] $\mc{L}^s_X=\bor(\{ X_i \}_{i\in I}),$
\item[$(4)$] $\mc{L}^{wo}_X\subseteq \sigma_X$. 
\end{enumerate}
\end{fact}

As in Proposition \ref{morph}, we can characterize the morphisms between definable spaces.
From Corollary \ref{homeo}, we get
\begin{cor} If two definable spaces over $(\mc{M}, \sigma )$
are isomorphic in $\lss(M)$, then they are definably homeomorphic.
\end{cor}

\begin{fact}[cf. \cite{DK}, Prop. I.1.3 and \cite{Pie3}, Fact 3.4.3]
Assume two objects $((X,\mc{L}_X), \mc{O}_X)$ and $((Y,\mc{L}_Y), \mc{O}_Y)$ of $\lds(\mc{M},\sigma)$ are given.
 A mapping $f:X\to Y$   is a morphism of $\lds(\mc{M},\sigma)$ iff the
following conditions are satisfied:
\begin{enumerate}
\item[a)] $f$ is bounded (i. e. $f(\mc{L}_X)$ is a refinement of $\mc{L}_Y$),
\item[b)] $f$ is $(\sigma_X,\sigma_Y)$-continuous,
\item[c)] $f$ is piecewise definable
(i. e. if $U\in \mc{L}_X$ and $V\in \mc{L}_Y$ is such that $f(U)\subseteq V$, then
the restriction  $f|^V_U:U\to V$ is $\mc{M}$-definable).
\end{enumerate}
\end{fact}

\begin{defi}[cf. \cite{Pie3}, Def. 3.3.7 and 3.4.4] 
Assume a locally definable space $((X,\mc{L}_X),\mc{O}_X)$ over $(\mc{M}, \sigma )$ is given as  the admissible union of some
affine definable spaces $(({X}_i, \mc{L}_{i}),\mc{O}_{X_i}), \: i\in I$. If $Y\subseteq X$ is {locally definable}, then it induces a  \textbf{subspace} $((Y,\mc{L}_Y), \mc{O}_Y)$ of  $((X,\mc{L}_X),\mc{O}_X)$  {in} $\lds(\mc{M},\sigma)$, given by the corresponding admissible union of 
$(({Y}_i,\mc{L}_{i}\cap_1 Y),\mc{O}_{Y_i})$, where $Y_i=X_i\cap Y, \: i\in I$.
\end{defi}

\begin{fact}[cf. \cite{Pie3}, Thm. 3.3.14]
For each structure with a topology $(\mc{M},\sigma)$, the categories 
$\ads(\mc{M},\sigma)$ and $\ds(\mc{M},\sigma)$ are concretely finitely complete.
\end{fact}

\begin{fact}[cf. \cite{Pie3}, Fact 3.4.10 and Thm. 3.4.11]
For each structure with a topology $(\mc{M},\sigma)$, the category $\lds(\mc{M},\sigma)$ has concrete finite limits and concrete coproducts.
\end{fact}

\begin{rem}[cf. \cite{Pie3}, Subsections 3.3 and 3.4] The following facts are well known:
\begin{enumerate}
\item[1)] Concrete finite coproducts exist in $\ds(\mc{M},\sigma)$.
\item[2)] Often also concrete finite coproducts exist in $\ads(\mc{M},\sigma)$, but not always (as in the case of $M$ a singleton).
\item[3)] Even finite products in $\lss(M)$ of affine definable spaces may  not be affine definable spaces.
\item[4)]  A coequalizer  even  in $\ads(\mc{M},\sigma)$ may  not be  concrete.
\end{enumerate}
\end{rem}

\begin{rem}
All of the results from the monograph \cite{DK} about locally semialgebraic spaces and from the paper \cite{Pie1} about locally definable spaces over o-minimal expansions of fields hold true when the definition of a generalized topological space by H. Delfs and M. Knebusch (with axioms i) --- viii))  is replaced with Definition \ref{lss} of a locally small space.
\end{rem}

\textbf{Acknowledgment.}
I thank Krzysztof Nowak for turning my attention to the monograph \cite{BGR}.


\end{document}